\documentclass[11pt,longtable,letterpaper,oneside,reqno]{amsart}
\usepackage{amsmath,amsfonts,amsthm,amssymb,amscd,stmaryrd,url,hyperref,subfig,cancel,todonotes}

\usepackage{amsrefs}
\newenvironment{rezabib}
  {\bibdiv\biblist\setupbib}
  {\endbiblist\endbibdiv}
  
  \def\setupbib{\catcode`@=\active}
\begingroup\lccode`~=`@
  \lowercase{\endgroup\def~}#1#{\gatherkey{#1}}
\def\gatherkey#1#2{\gatherkeyaux{#1}#2\gatherkeyaux}
\def\gatherkeyaux#1#2,#3\gatherkeyaux{\bib{#2}{#1}{#3}}  

\usepackage{longtable,tabularx,tabulary}
\usepackage{titletoc}
\usepackage{float,mathtools}
\usepackage{color}
\usepackage{anysize,placeins,enumitem}
\usepackage{latexsym}
\usepackage{graphicx}
\usepackage{graphicx}
\usepackage{epsfig}
\usepackage{booktabs}
\usepackage{xcolor}
\usepackage{textgreek}
\usepackage{textcomp}
\usepackage{cases}

\usepackage{todonotes}
\usepackage{mathrsfs}
\usepackage{caption}
\newcommand{\C}{{\mathbb{C}}}

\renewcommand{\Re}{{\mathfrak{Re}}}
\renewcommand{\Im}{{\mathfrak{Im}}}

 \DeclareMathOperator{\modd}{mod}

\newcommand{\g}{\gamma}

\newtheorem{theorem}{Theorem}[section]
 \newtheorem{corollary}[theorem]{Corollary}
 \newtheorem{lemma}[theorem]{Lemma}
 
 \newtheorem{proposition}[theorem]{Proposition}
 \newtheorem{defn}[theorem]{Definition}

  \theoremstyle{remark}

\numberwithin{equation}{section}
\reversemarginpar 

\begin{document}

\title[Counting Zeros of the Riemann Zeta Function]{Counting Zeros of the Riemann Zeta Function}

\author[E. Hasanalizade]{Elchin Hasanalizade}
\address{Department of Mathematics and Computer Science\\
University of Lethbridge\\
4401 University Drive\\
Lethbridge, Alberta\\
T1K 3M4 Canada}
\email{e.hasanalizade@uleth.ca}

\author[Q. Shen]{Quanli Shen}
\address{Department of Mathematics and Computer Science\\
University of Lethbridge\\
4401 University Drive\\
Lethbridge, Alberta\\
T1K 3M4 Canada}
\email{quanli.shen@uleth.ca}

\author[P.J. Wong]{Peng-Jie Wong}
\address{National Center for Theoretical Sciences\\
No. 1, Sec. 4, Roosevelt Rd., Taipei City, Taiwan}
\email{pengjie.wong@ncts.tw}

\thanks{This research was supported by the NSERC Discovery grants RGPIN-2020-06731 of Habiba Kadiri and  RGPIN-2020-06032 of Nathan Ng. P.J.W. is currently an NCTS postdoctoral fellow; he was supported by a PIMS postdoctoral fellowship and the University of Lethbridge during part of this research.}

\begin{abstract}
In this article, we show that
$$
\left| N (T)  - \frac{T}{  2 \pi} \log \left( \frac{T}{2\pi e}\right)  \right|
 \le    0.1038  \log T + 0.2573  \log\log T + 9.3675
$$
where  $N(T)$ denotes the number of non-trivial zeros $\rho$, with $0<\Im(\rho) \le T$, of the Riemann zeta function. This improves the previous result of Trudgian for sufficiently large $T$.  The improvement comes from the use of various subconvexity bounds and ideas  from the work of Bennett \emph{et al.} on counting zeros of Dirichlet $L$-functions.
\end{abstract}




\maketitle

\section{Introduction}
Let $\zeta(s)$ be the Riemann zeta function defined by 
\[
\zeta (s) = \sum_{n=1}^{\infty} \frac{1}{n^s},
\]
for $\Re(s)>1$, which has an analytic continuation to a meromorphic function on $\Bbb{C}$ with only a simple pole at $s = 1$. The study of zeros of $\zeta(s)$ is an important topic in number theory. In this article, we shall estimate the number of non-trivial zeros $\rho=\beta+i\gamma$, with $0<\gamma \le T$, of $\zeta(s)$. For $T\geq 0$,   we set 
\[
N(T) = \# \{ \rho \in  \C  \mid  \zeta(\rho) =0,\  0  <\beta <  1, \ 0<\gamma \leq T\}.
\]

Before stating our results, we shall note that the study of $N(T)$ has a long history. Indeed, for $T\ge T_0$, writing 
\begin{equation}\label{his-bound}
\left| N (T)  - \frac{T}{   2 \pi} \log \left( \frac{T}{2\pi e}\right)  \right|
 \le   C_1 \log T  + C_2   \log\log T  + C_3,
\end{equation}
we have the following table summarising the progress that has been made. 

\begin{table}[htbp] 
\centering
\begin{tabular}{ |c||c|c|c|c|   } 
 \hline  
  &  $C_1$ & $C_2$ & $C_3$ & $T_0$ \\ 
 \hline
 Von Mangoldt \cite{vMo05} (1905) & 0.4320   & 1.9167  &  13.0788  &  28.5580     \\ 
 \hline
 Grossmann \cite{Gr13} (1913) & 0.2907   & 1.7862 &  7.0120 &  50  \\ 
 \hline 
 Backlund \cite{Ba18} (1918)  & 0.1370   & 0.4430  & 5.2250 & 200 \\ 
 \hline 
 Rosser \cite{Ro41} (1941)  & 0.1370   & 0.4430  & 2.4630 & 2\\ 
 \hline   
 Trudgian \cite{Tr14-2} (2014)  & 0.1120   & 0.2780  & 3.3850 & $e$\\ 
 \hline 
 Corollary \ref{main-thm}   & 0.1038   & 0.2573  & 9.3675 & $e$\\ 
 \hline  
\end{tabular}
   \caption{Explicit bounds for $N(T)$ in \eqref{his-bound}}\label{table0} 
\end{table}

The importance of explicit bounds for $N(T)$ comes from the fact that they are crucial for estimating sums over zeros of $\zeta(s)$, and all the best known bounds for $\pi(x)$ and $\psi(x)$ rely on them (see, e.g., \cite{FK15}).

In this article, we prove the following general result for $N(T)$ with explicit dependence on the given bounds for $\zeta(s)$ on  the both $\frac{1}{2}$-line and 1-line. 

\begin{theorem}\label{main-thm-0}
Let $c,r,\eta$ be positive real numbers satisfying 
$$
-\frac{1}{2}<c-r<1-c < -\eta  <\frac{1}{4} \le \delta:=2c- \sigma_1 -\frac{1}{2} < \frac{1}{2} < 1+\eta < \sigma_1 :=c+ \frac{(c-1/2)^2}{r}< c+r
$$
and $\theta_{1 + \eta} \leq 2.1$, where $\theta_y$ is defined in \eqref{def-theta}. Let $c_1,c_2,k_1,k_3\ge 0$, $k_2\in [0,\frac{1}{2}]$ and $t_0, t_1\ge e$ such that for $t\ge t_0$,
\begin{equation}\label{cond-1}
|\zeta (1+ it )| \leq c_1 (\log t)^{c_2},
\end{equation}
and for $t\ge t_1$,
\begin{equation}\label{cond-2}
|\zeta (\tfrac{1}{2} + it)| \leq k_1 t^{k_2} (\log t)^{k_3}.
\end{equation}
Let $T_0\ge e$ be fixed. Then for any $T\ge T_0$, we have
\begin{equation}\label{main-bound-0}
\left| N (T)  - \frac{T}{   2 \pi} \log \left( \frac{T}{2\pi e}\right) +\frac{1}{ 8 } \right|
 \le   C_1 \log T  +C_2   \log\log T  +C_3,
\end{equation}
where  $C_1=C_1(c,r,\eta;k_2)$, $C_2=C_2(c,r,\eta;c_2,k_3) $, $C_3=C_3(c,r,\eta;c_1,c_2,t_0 ,k_1,k_2,k_3,t_1;T_0)$ are defined in \eqref{def-Cj}, \eqref{def-C1}, \eqref{def-C2}, and \eqref{def-C3}, and some admissible values of $C_1$, $C_2$, and $C_3$ are recorded in Table \ref{table2} in Section \ref{final}.
\end{theorem}

As a consequence, we obtain an  explicit  estimate for $N(T)$ as follows.

\begin{corollary}\label{main-thm}
For any $T\ge e$, we have
\begin{equation}\label{main-bound-1}
\left| N (T)  - \frac{T}{   2 \pi} \log \left( \frac{T}{2\pi e}\right)  \right|
 \le    0.1038  \log T + 0.2573  \log\log T + 9.3675.
\end{equation}
\end{corollary}

Note that $C_1 = 0.1038$ is the smallest value that can be obtained by our argument and computation, and one can make $C_2$ and $C_3$ smaller at expense of larger $C_1$.

Let  $S(T) = \frac{1}{\pi} \Delta_L \arg\zeta(s)$, where  $L$ denotes the straight line from $2$ to $2 + iT$ and then to $\frac{1}{2} + iT$. We also have the following theorem concerning the argument of $\zeta(s)$ along the critical line.

\begin{theorem}\label{main-thm-S(T)} In the notation and assumptions of Theorem \ref{main-thm-0}, for any $T\ge T_0$, we have
\begin{equation}\label{main-bound-2}
|S(T)|\le C_1 \log T +  C_2 \log\log T + C'_3,
\end{equation}
where  $C_1=C_1(c,r,\eta;k_2)$, $C_2=C_2(c,r,\eta;c_2,k_3) $, $C'_3=C'_3(c,r,\eta;c_1,c_2,t_0 ,k_1,k_2,k_3,t_1;T_0)$ are defined in \eqref{def-Cj}, \eqref{def-C1}, \eqref{def-C2}, and \eqref{def-D3}, and some admissible values of $C_1$, $C_2$, and $C'_3$ are recorded in Table \ref{table2} in Section \ref{final}. 
\end{theorem}

A stronger bound for $S(T)$, up to certain given height, can be confirmed using the database of non-trivial zeros of $\zeta(s)$ computed by Platt and made available at \cite{LMFDB}. Indeed, nowadays, one has
\begin{equation}\label{bd-S}
|S(T)| \le 2.5167
\end{equation}
for $0 \leq T \leq 30\,610\,046\,000$.\footnote{Recently,  Platt and
Trudgian \cite{PT21} verified the Riemann hypothesis for the height up to $3 \cdot 10^{12}$, which would allow one to further bound $S(T)$ for $0 \leq T \leq 3 \cdot 10^{12}$.}
Hence, by Theorem \ref{main-thm-S(T)} (with $T_0 = 30\,610\,046\,000$ and Table \ref{table2}) and \eqref{bd-S}, we derive the following explicit bound for $S(T)$. 

\begin{corollary}\label{main-thm-2}
For any $T\ge e$, we have
$$
|S(T)|\le \min\{  0.1038  \log T + 0.2573  \log\log T + 8.3675,\enspace 0.1095\log T +   0.2042\log\log T +   3.0305 \}.
$$
\end{corollary}

We note that this improves the previous best-known explicit bound (for $T$ sufficiently large) due to Platt and Trudgian \cite{PT} (see also \cite{Tr12, Tr14-2}), who showed that for $T\ge e$,
$$
|S(T)|\le 0.110\log T+ 0.290\log\log T + 2.290.
$$

The proofs of Theorems \ref{main-thm} and \ref{main-thm-S(T)} are based on the work of \cite{BMOR20,HSW,Tr14-2,Tr15}.\footnote{Regrettably, as pointed out in \cite{BMOR20},  there is an error appearing in \cite{Tr14-2,Tr15} (and \cite{PT}, where the erroneous result of \cite{Tr14-2} was used) since  the ranges of  parameters involved in the final formulae were not verified  properly; \cite{BMOR20,HSW} fix this issue for \cite{Tr15}. In a certain degree, the objective of the presented paper is to fix the error occurring in \cite{Tr14-2}.}  Compared to the considerations of Dirichlet $L$-functions in \cite{BMOR20,Tr15} and Dedekind zeta functions in \cite{HSW,Tr15}, we further use subconvexity bounds for $\zeta(\frac{1}{2}+it)$, together with the Phragm\'{e}n-Lindel\"{o}f principle and the functional equation, to obtain a shaper estimate for $\zeta(s)$ in the strip $0\le \Re(s)\le 1$. Also, based on the idea of  \cite{BMOR20}, we refine the bound for $\zeta(s)$ on $\Re(s)<0$ used in \cite{Tr14-2} by  bounding $\zeta(s)$ over  $\Re(s)<-\frac{1}{2}$. Lastly, we note that most numerical computations were performed in Maple.

\section{Main term and bounds for gamma factors}\label{main}

We recall the completed Riemann zeta function $\xi(s)$ is defined by 
\begin{equation}\label{def-xi}
\xi(s)= s(s-1) \g (s) \zeta (s),
\end{equation}
where 
$$
\gamma (s) =  \pi^{-\frac{s}{2}} \Gamma \left( \frac{s}{2}\right).
$$
It is well-known that $\xi(s)$ can be extended to an entire function of order 1, which satisfies the functional equation
\begin{equation}\label{FE}
\xi(s) = \xi (1-s).
\end{equation}

To follow the argument used in \cite{HSW}, it would be simpler to work with the following ``symmetric version'' of $N(T)$. We introduce 
\[
N_{\Bbb{Q}}(T) = \# \{ \rho \in  \C  \mid  \zeta(\rho) =0,\  0  <\beta <  1, \ |\gamma| \leq T\}
\]
for $T\geq 0$.\footnote{This notation agrees with  \cite{HSW}, where we defined $N_K(T)$, the zero-counting function for the Dedekind zeta function of a number field $K$.} Note that $N_{\Bbb{Q}}(T)= 2 N(T)$.

Let $\sigma_1 >1$, and let $\mathcal{R}$ be the rectangle with vertices $\sigma_1 - iT,\  \sigma_1 + iT,\  1-  \sigma_1+ iT$, and $1- \sigma_1 - iT$ (that is away from zeros of $\xi (s)$). Since $\xi (s)$ is entire, applying the  argument principle, we know that 
\[
 N_{\Bbb{Q}}(T) = \frac{1}{2\pi } \Delta_{\mathcal{R}} \arg \xi (s).
\]
 We let $\mathcal{C}$ be the part of the contour of $\mathcal{R}$ in $\Re (s ) \geq \frac{1}{2}$ and $\mathcal{C}_0$ be the part of the contour of $\mathcal{R}$ in $\Re (s ) \geq \frac{1}{2}$ and $\Im (s) \geq 0$. From the functional equation \eqref{FE} and the fact that $\overline{\xi(s)} = \xi (\bar{s})$, it follows that
\[
\Delta_{\mathcal{R}} \arg \xi(s) = 2 \Delta_{\mathcal{C}} \arg \xi(s)  =4    \Delta_{\mathcal{C}_0} \arg \xi(s),
\]
and thus
\begin{align}\label{formula-N-K}
 N_{\Bbb{Q}}(T)    = \frac{2}{\pi}   \Delta_{\mathcal{C}_0} \arg \xi (s).
\end{align}

Now, by \eqref{def-xi}, we arrive at
\begin{align}\label{deta-expansion}
 \begin{split}
\Delta_{\mathcal{C}_0} \arg \xi(s) 
= \Delta_{\mathcal{C}_0} \arg s + \Delta_{\mathcal{C}_0} \arg \pi^{-s/2}  
+  \Delta_{\mathcal{C}_0} \arg \Gamma  \left(\frac{s}{2} \right)  + \Delta_{\mathcal{C}_0} \arg \left( (s-1) \zeta (s)\right).
 \end{split}
\end{align}
In addition, by a straightforward calculation, we have
\begin{align}   \label{deta-explicit}
 \begin{split}
&\Delta_{\mathcal{C}_0} \arg s = \arctan (2T),\\
&\Delta_{\mathcal{C}_0} \arg \pi^{-s/2} = \frac{T}{2} \log \left( \frac{1}{\pi }\right),\\
&\Delta_{\mathcal{C}_0}  \arg \Gamma (s) = \Delta_{\mathcal{C}_0}  (\Im \log \Gamma (s)) = \Im \log \Gamma \left(\frac{1}{2} + iT \right) .
 \end{split}
\end{align}

In order to control the contribution of the gamma factor in \eqref{deta-expansion}, we set
\begin{equation}\label{def-g}
g(T)= \frac{2}{\pi}\Im \log \Gamma \left(\frac{1}{4}  + i\frac{T}{2} \right) - \frac{T}{\pi} \log \left( \frac{T}{2e}\right) +\frac{1}{4}
\end{equation}
and recall that by \cite[Proposition 3.2]{BMOR20}, one has  
\begin{equation}\label{bd-g}
|g(T)|\le \frac{1}{25T}
\end{equation}
for $T\ge 5/7$.
Now, gathering  \eqref{formula-N-K}, \eqref{deta-expansion}, \eqref{deta-explicit}, and \eqref{def-g}, we establish
\begin{equation}\label{bd-N-K-1}
 N_{\Bbb{Q}}(T) = \frac{2}{\pi}\arctan ( 2T) + g(T) +  \frac{T}{\pi} \log  \left( \frac{T}{2\pi e}\right)- \frac{1}{4} 
+ \frac{2}{\pi}\Delta_{\mathcal{C}_0} \arg ( (s-1) \zeta(s) ).
\end{equation}

To control $\Delta_{\mathcal{C}_0} \arg ( (s-1) \zeta(s) )$, we let $\mathcal{C}_1$ be the vertical line from $\sigma_1$ to $\sigma_1 + iT$ and $\mathcal{C}_2$ be the horizontal line from $\sigma_1 + iT$ to $\frac{1}{2} + iT$. As
\begin{align*}
\Delta_{\mathcal{C}_1} \arg (s-1)\zeta (s)
 &=\Delta_{\mathcal{C}_1}  \arg (s-1) +\Delta_{\mathcal{C}_1}  \arg \zeta (s) 
= \arctan \left( \frac{T}{\sigma_1 -1}\right) +\Delta_{\mathcal{C}_1} \arg \zeta (s)
\end{align*}
and for $\sigma_1 > 1$,
$$
|\Delta_{\mathcal{C}_1}  \arg  \zeta (s)|
 =  |\arg \zeta (\sigma_1+iT) |
  \leq |\log \zeta (\sigma_1+iT) |
  \le  \log \zeta (\sigma_1),  
$$
we obtain
\begin{equation}\label{bd-C1}
|\Delta_{\mathcal{C}_1} \arg (s-1)\zeta (s)| \leq \frac{\pi}{2} +   \log \zeta(\sigma_1).
\end{equation}
Hence, from \eqref{bd-g}, \eqref{bd-N-K-1}, and \eqref{bd-C1}, it follows that
\begin{equation}\label{main-est}
 \left|  N_{\Bbb{Q}}(T)  -  \frac{T}{\pi} \log\left( \frac{T}{2\pi e}\right)  +  \frac{1}{4} \right|  \leq  
 2 +\frac{1}{25 T}+ \frac{2}{\pi} \log \zeta(\sigma_1) +  \frac{2}{\pi}| \Delta_{\mathcal{C}_2}  \arg ( (s-1) \zeta(s) )|.
\end{equation}

To end this section, we shall borrow some estimates for the gamma function from \cite{BMOR20} as follows. For  $0\le d < 9/2$ and $T\ge 5/7$, we define
$$
\mathcal{E}(T,d)=  \left| \left.  \Im \log \Gamma \left( \frac{\sigma +iT}{2} \right)\right|_{\sigma=\frac{1}{2}}^{\frac{1}{2}+d} 
+  \left. \Im \log \Gamma \left( \frac{\sigma +iT}{2} \right)\right|_{\sigma=\frac{1}{2}}^{\frac{1}{2}-d} \right|,
$$
As in \cite[p. 1463]{BMOR20}, we set
\begin{align*}
 \begin{split}
E (T,d)& = \frac{2T /3}{(2d + 17)^2 + 4 T^2} + \frac{2T/3}{( -2d +17)^2 + 4T^2}  - \frac{4T/3}{17^2 + 4T^2}\\
& + \frac{T}{2} \log \left( 1 + \frac{ 17^2}{4T^2}\right) - \frac{T}{4} \log \left( 1 + \frac{( 2d +17)^2}{4T^2}\right) - \frac{T}{4} \log \left( 1 + \frac{( -2d + 17)^2}{4T^2}\right)\\
&+ \frac{(8+6\pi)/45}{(( 2d + 17)^2 + 4T^2)^{3/2}}
 +  \frac{(8+6\pi)/45}{(( - 2d + 17)^2 + 4T^2)^{3/2}} + \frac{2(8+ 6\pi)/45}{( 17^2 + 4T^2)^{3/2}} \\
& + \sum_{k=0}^3\left( 2\arctan \frac{1+4k}{2T} - \arctan \frac{2d+1+4k}{2T} - \arctan \frac{-2d+1+4k}{2T} \right)\\  
& + \frac{ 2d +15}{4} \arctan \frac{ 2d + 17}{2T}  + \frac{ - 2d +15}{4} \arctan \frac{ -2d + 17}{2T}
 - \frac{ 15}{2} \arctan\frac{17}{2T}.
 \end{split} 
\end{align*}
It has been shown in \cite[p. 1462]{BMOR20} that  $\mathcal{E}(T,d) \le E(T,d)$ for $0 \leq d < 9/2$ and  $T\ge 5/7$. Moreover, one has the following lemma established in \cite[Lemma 3.4]{BMOR20}.

\begin{lemma}\label{EK-bd-final}
For $0\le \delta_1 \le d < 9/2$ and  $T\ge 5/7$,  one has
$$
0<E(T,\delta_1)\le E(T,d).
$$
Also, for $d\in [\frac{1}{4}, \frac{5}{8}] $ and $T\ge 5/7$, one has
$$
\frac{E(T,d)}{\pi} \le \frac{640d-112}{1536(3T-1)} + \frac{1}{2^{10}}.
$$
\end{lemma}

We remark that the number $5/7$ above is not random but is borrowed directly from \cite{BMOR20}.  In \cite{BMOR20}, it is chosen to obtain a proper numerical bound (in order to count the zeros of Drirchlet $L$-functions). One may replace $5/7$ with an appropriate larger $T_0$ to get a better result.

\section{Backlund's trick}\label{B-trick}

In order to  estimate $\Delta_{\mathcal{C}_2}  \arg ( (s-1) \zeta(s) )$, we shall borrow some results from \cite{HSW}. Define
\begin{align}\label{def-fN}
f_N(s)=\frac{1}{2}\left(( (s+iT-1)\zeta(s+iT))^N+ ((s-iT-1)\zeta(s-iT))^N\right)
\end{align}
for  $N\in\Bbb{N}$. Let $D(c,r)$ be the open disk centred at $c$ with radius $r$. For any $N\in  \Bbb{N}$, we define
\begin{align*}
S_N(c,r)=\frac{1}{N}\sum_{z\in\mathcal{S}_N(D(c,r) )  }\log\frac{r}{|z-c|},
\end{align*}
where $\mathcal{S}_N(D(c,r) )$ denotes the set of  zeros of $f_N(s)$ in $D(c,r)$. In \cite[Proposition 3.5]{HSW},  the authors  prove the following  upper bound for $S_N(c,r)$.

\begin{proposition}\label{Je-bd}
Let $c$, $r$, and $\sigma_1$ be real numbers such that
\begin{align*}
c-r<\frac{1}{2}<1<c<\sigma_1<c+r.
\end{align*}
Let  $F_{c,r}:[-\pi,\pi]\to\mathbb{R}$ be an even function such that $F_{c,r}(\theta)\geq\frac{1}{N_m}\log|f_{N_m}(c+re^{i\theta})|$. Then there is an infinite sequence of natural numbers $(N_m)_{m=1}^{\infty}$ such that
\begin{align*}
\limsup_{m\to\infty}S_{N_m}(c,r)\leq\log\left(\frac{1}{\sqrt{(c-1)^2+T^2}}\frac{\zeta(c)}{\zeta(2c)}\right)+\frac{1}{\pi}\int_{0}^{\pi}F_{c,r}(\theta)d\theta.
\end{align*}
\end{proposition}

To end this section, we recall the following version of Backlund's trick established in \cite[Proposition 3.7]{HSW} (cf.  \cite{BMOR20, Tr14-2, Tr15}).  As explained in \cite[Sec. 3]{Tr15}, using  Backlund's trick, one can track contribution from zeros of $f_N(\sigma)$ in $[1-\sigma_1, \frac{1}{2} ]$  whence there are zeros  in $[\frac{1}{2}, \sigma_1 ]$.  

\begin{proposition}[Backlund's trick]\label{BT}
Let $c$ and $r$ be real numbers. Set 
\[
\sigma_1 = c + \frac{(c-1/2)^2}{r} \quad \text{and} \quad  \delta = 2c - \sigma_1 -\frac{1}{2}.
\]
If $1 < c < r$ and $0 < \delta < \frac{1}{2}$, then
\[ 
\left| \left.  \arg \left( (\sigma + iT-1)\zeta(\sigma + iT) \right) \right|_{\sigma = \sigma_1}^{1/2} \right|
 \leq \frac{\pi S_N(c,r)}{2 \log (r/ (c-1/2))} + \frac{E(T,\delta)}{2} +\frac{\pi}{N}+ \frac{\pi}{2N}+
\frac{\pi}{4}.
\]
\end{proposition}

\section{Convexity and subconvexity bounds  and $F_{c,r}(\theta) $}

\subsection{Convexity and subconvexity bounds}\label{CSCB}
In light of Propositions \ref{Je-bd}  and \ref{BT}, to estimate \eqref{main-est}, we shall construct an appropriate $F_{c,r}(\theta)$. We first recall the following version of the Phragm\'{e}n-Lindel\"{o}f principle established by Trudgian \cite[Lemma 3]{Tr14-2}.
\begin{proposition}[Phragm\'{e}n-Lindel\"{o}f principle]\label{PL-p}
Let $a,b,Q$ be real numbers such that $b>a$ and $Q+a>1$. Let $f(s)$ be a holomorphic function on the strip $a\leq\Re(s)\leq b$ such that 
\begin{align*}
|f(s)|<C \exp(  e^{k|t| })
\end{align*}
for some $C>0$ and $0<k<\frac{\pi}{b-a}$. Suppose, further, that there are  $A,B,\alpha_1, \alpha_2, \beta_1, \beta_2\ge 0$ such that $\alpha_1\geq\beta_1$  and
\[
|f(s)|\leq
\begin{cases}
A|Q+s|^{\alpha_1}(\log{|Q+s|})^{\alpha_2} & \text{for $\Re(s)=a$;}\\
B|Q+s|^{\beta_1}(\log{|Q+s|})^{\beta_2} &\text{for $\Re(s)=b$.}
\end{cases}
\]
Then for $a\leq\Re(s) \leq b$, one has
\begin{align*}
|f(s)|\leq\big\{A|Q+s|^{\alpha_1}(\log{|Q+s|})^{\alpha_2}\big\}^{\frac{b-\Re(s)}{b-a}}\big\{B|Q+s|^{\beta_1}(\log{|Q+s|})^{\beta_2}\big\}^{\frac{\Re(s)-a}{b-a}}.
\end{align*}
\end{proposition}

We shall assume that there are $c_1,c_2,k_1,k_3\ge 0$, $k_2\in [0,\frac{1}{2}]$ and $t_0, t_1\ge e$ such that
\begin{align}\label{equ:zeta-bd-1}
|\zeta (1+ it )| \leq c_1 (\log t)^{c_2}
\end{align}
for $t\ge t_0$, and 
\begin{align}\label{equ:zeta-bd-1/2}
|\zeta (\tfrac{1}{2} + it)| \leq k_1 t^{k_2} (\log t)^{k_3}
\end{align}
for $t\ge t_1$. Recall that $0<\eta\le \frac{1}{2}$. Now, we split our consideration into the following six cases.

(1) Assume  $\sigma \geq 1+ \eta$. The trivial bound for the zeta function immediately gives 
\begin{align}
|\zeta(s)| \leq \zeta(\sigma).
\label{equ:zeta-bd-1+eta}
\end{align}

(2) Assume $1 \leq \sigma \leq 1+ \eta $. From \eqref{equ:zeta-bd-1}, it follows that there is $Q_0> 0$ such that
\begin{align}\label{equ:zeta-s-bd-1}
|(1+it-1) \zeta (1+it)| \leq c_1 |Q_0 + (1 + it)| (\log |Q_0 + (1+it)|)^{c_2}
\end{align}
for all $t$. (We note that $Q_0$ depends on $c_1,c_2,t_0$, and it can be computed explicitly.) Also,  by \eqref{equ:zeta-bd-1+eta}, it is clear that
$$
|(1+\eta+it-1) \zeta (1+\eta+it)| \le \zeta(1+\eta) |Q_0 + (1 +\eta + it)|. 
$$
Thus, by Proposition \ref{PL-p}, for $1\leq \sigma \leq 1+ \eta$,
\begin{align*}
|(s-1) \zeta(s) | \leq \left(c_1 |Q_0 + s| (\log |Q_0 +s|)^{c_2}  \right)^{\frac{1+\eta- \sigma}{\eta}} \left( \zeta (1+ \eta) |Q_0 + s| \right)^{\frac{\sigma - 1}{\eta}},
\end{align*}
which implies 
\begin{align*}
| \zeta(s) | \leq \frac{1}{|s-1|}\left(c_1 |Q_0 + s| (\log |Q_0 +s| )^{c_2} \right)^{\frac{1+\eta- \sigma}{\eta}} \left( \zeta (1+ \eta) |Q_0 + s| \right)^{\frac{\sigma - 1}{\eta}}.
\end{align*}

(3) Let $\frac{1}{2}\leq \sigma \leq 1$. Using \eqref{equ:zeta-bd-1/2}, we deduce 
\begin{align}\label{equ:bd-zeta-half}
|(\tfrac{1}{2} + it -1)\zeta(\tfrac{1}{2} + it)| \leq k_1 |Q_1 + (\tfrac{1}{2} + it)|^{k_2+1} (\log |Q_1 + (\tfrac{1}{2} + it)|)^{k_3},
\end{align}
for all $t$, where $Q_1 >0$ is a constant depending only on $k_1,k_2,k_3,t_1$ (which can be computed  directly). Now, by  \eqref{equ:zeta-s-bd-1}, \eqref{equ:bd-zeta-half}, and Proposition \ref{PL-p}, for $\frac{1}{2}\leq \sigma \leq 1$, we have 
\begin{align*}
| \zeta(s) | \leq \frac{1}{|s-1|}\left( k_1 |Q_2 + s|^{k_2 + 1 } (\log |Q_2 +s|)^{k_3}  \right)^{2- 2\sigma} \left( c_1 |Q_2 + s| (\log |Q_2 +s|)^{c_2} \right)^{2\sigma - 1},
\end{align*}
where $Q_2=\max\{Q_0,Q_1\}$.

(4) Assume $0 \leq \sigma \leq \frac{1}{2}$. On the one hand, by \eqref{equ:zeta-bd-1/2}, there is $Q_3 >0$ such that 
\begin{align}\label{equ:bd-zeta-half'}
|\zeta(\tfrac{1}{2} + it)| \leq k_1 |Q_3 + (\tfrac{1}{2} + it)|^{k_2} (\log |Q_3 + (\tfrac{1}{2} + it)|)^{k_3}
\end{align}
for all $t$.
On the other hand, as \eqref{FE} gives
\begin{equation}\label{FE2}
\zeta (s) = \pi^{s - \frac{1}{2}} \frac{\Gamma(\frac{1}{2} -\frac{s}{2})}{\Gamma (\frac{s}{2})} \zeta (1-s),
\end{equation}
it follows from \eqref{equ:zeta-bd-1} and the estimate 
\begin{equation}\label{gamma-quo-bd}
\left| \frac{\Gamma (\frac{1}{2} - \frac{s}{2} )}{\Gamma (\frac{s}{2})}  \right|
\leq \left( \frac{|1 + s|}{2}\right)^{\frac{1}{2} - \sigma},
\end{equation}
for $-\frac{1}{2} \leq \sigma \leq \frac{1}{2}$,
that
\[
|\zeta(it)| \leq c_1\left( \frac{|1+ it|}{2\pi}\right)^{\frac{1}{2}} (\log t)^{c_2}
\]
for $t \geq t_0$. Hence, checking small values of $t$, we can find $Q_4\ge 1$ such that
\begin{align}\label{equ:bd-zeta-0}
|\zeta(0+ it)| \leq \frac{c_1}{\sqrt{2\pi}} |Q_4 + it |^{\frac{1}{2}} (\log |Q_4 + it|)^{c_2}
\end{align}
for all $t$.  Now, together with \eqref{equ:bd-zeta-half'},   by Proposition \ref{PL-p}, we then obtain
\begin{align*}
| \zeta(s) | \leq \left( \frac{c_1}{\sqrt{2\pi}} |Q_5 + s|^{\frac{1}{2} } (\log |Q_5 +s|) ^{c_2} \right)^{1-2\sigma} \left( k_1 |Q_5 + s| ^{k_2 }(\log |Q_5 +s|)^{k_3} \right)^{2\sigma}.
\end{align*}
where $Q_5=\max\{Q_3,Q_4\}$.

(5)   For $- \eta \leq \sigma \leq 0$, using \eqref{equ:zeta-bd-1+eta}, the functional equation \eqref{FE2}, and the Gamma bound \eqref{gamma-quo-bd}, as $Q_4\ge 1$, we obtain
\[
| \zeta(-\eta + it) | 
\leq  \left( \frac{|1- \eta + it|}{2\pi} \right) ^{\frac{1}{2} + \eta} \zeta( 1 + \eta)
\le  \left( \frac{|Q_4- \eta + it|}{2\pi} \right) ^{\frac{1}{2} + \eta} \zeta( 1 + \eta)
\] 
for all $t$. This estimate, combined with \eqref{equ:bd-zeta-0} and Proposition \ref{PL-p}, then yields
\begin{align*}
| \zeta(s) | \leq \left( \frac{1}{(2\pi)^{\frac{1}{2} + \eta}} \zeta(1 + \eta)|Q_4 + s|^{\frac{1}{2}  + \eta}   \right)^{\frac{-\sigma}{\eta}} \left( \frac{c_1}{\sqrt{2\pi}} |Q_4 + s| ^{\frac{1}{2}}(\log |Q_4 +s|)^{c_2} \right)^{\frac{\sigma+ \eta}{\eta}}.
\end{align*}

(6) Finally, for $\sigma \leq -\eta$, we use \eqref{FE2} to deduce
\[
|\zeta (s)| \leq  \pi^{\sigma - \frac{1}{2}} \left| \frac{\Gamma(\frac{1}{2} -\frac{s}{2})}{\Gamma (\frac{s}{2})} \zeta (1-s) \right| 
\leq  \pi^{\sigma - \frac{1}{2}} \left| \frac{\Gamma(\frac{1}{2} -\frac{s}{2})}{\Gamma (\frac{s}{2})} \right|\zeta (1-\sigma).
\]
From the proof of \cite[Theorem 5.7]{BMOR20}, it follows that for $a,b \in \{0,1 \}$ and $k\in\Bbb{N}$,
\begin{align*}
\frac{\Gamma(\frac{a}{2}+\frac{1-s}{2})}{\Gamma(\frac{a}{2}+\frac{s}{2})}
=\frac{\Gamma(\frac{b}{2}+\frac{1-(s+k)}{2})}{\Gamma(\frac{b}{2}+\frac{s+k}{2})}2^{-k}\left(\prod_{j=1}^{k} (s+j-1)\right)\frac{\sin(\frac{\pi}{2}(s+k+1-b))}{\sin(\frac{\pi}{2}(s+1-a))}.
\end{align*}
Now, for $x\in\Bbb{R}$, we let $[x]$ be the integer closest to $x$ (if  there are two integers equally close to $x$, we then choose the one closer to 0). Note that for $a=0$ and $b\equiv k\  (\modd 2) $, the sine factors above are $\pm1$. Thus, upon taking $k=-[\sigma]$ and applying \cite[Lemmata 1 and 2]{Ra59} to bound the ratio $\Gamma(\frac{b}{2}+\frac{1-(s+k)}{2})/\Gamma(\frac{b}{2}+\frac{s+k}{2})$, we arrive at
\begin{align*}
\left|\frac{\Gamma(\frac{1-s}{2})}{\Gamma(\frac{s}{2})}\right|\leq 
\left(\frac{1}{2}|1+s-[\sigma]|\right)^{\frac{1}{2}+[\sigma]-\sigma}2^{[\sigma]} \left(\prod_{j=1}^{-[\sigma]} | s+j-1 |\right),
\end{align*}
which gives 
\[
|\zeta(s)| \leq \zeta(1-\sigma) \left( \frac{1}{2 \pi}\right)^{\frac{1}{2}- \sigma} \left(|1+s-[\sigma]|\right)^{\frac{1}{2}+[\sigma]-\sigma} \left(\prod_{j=1}^{-[\sigma]} | s+j-1 |\right) .
\]

\subsection{Constructing and estimating $F_{c,r}(\theta) $}

\subsubsection{Bounding $\frac{1}{N}\log|f_N(s)|$}

With the above convexity and subconvexity bounds in hand, we are in a position to bound $\frac{1}{N}\log|f_N(s)|$, where $f_N(s)$ is defined in \eqref{def-fN}. For $\sigma\geq1+\eta>1$, we have
\begin{align*}
|f_N(s)|&\leq\frac{1}{2}\left(|s+iT-1|^N|\zeta(s+iT)|^N+|s-iT-1|^N|\zeta(s-iT)|^N\right)\\
&\leq ((\sigma-1)^2+(|t|+T)^2 )^{\frac{N}{2}}\zeta(\sigma)^{{N}}.
\end{align*}
Taking logarithms and dividing both sides by $N$ gives
\begin{align*}
\frac{1}{N} \log |f_N(s)| &\leq \frac{1}{2} \log ((\sigma-1)^2+(|t|+T)^2) + \log \zeta(\sigma).
\end{align*}

For $1\leq \sigma \leq 1+\eta$, we have
\begin{align*}
|f_N(s)|&\leq \left(c_1\sqrt{(Q_0+\sigma)^2+(|t|+T)^2}\left(\log{\sqrt{(Q_0+\sigma)^2+(|t|+T)^2}}\right)^{c_2}\right)^{\frac{N(1+\eta-\sigma)}{\eta}}\\
&\times \left(\zeta(1+\eta)\sqrt{(Q_0+\sigma)^2+(|t|+T)^2}\right)^{\frac{N(\sigma-1)}{\eta}}.
\end{align*}
Taking logarithms of both sides and dividing by $N$, we obtain
\begin{align*}
\frac{1}{N}\log{|f_N(s)|}&\leq \frac{1+\eta-\sigma}{\eta}\log{\frac{c_1}{2^{c_2}}}+\frac{\sigma-1}{\eta}\log{\zeta(1+\eta)}+\frac{1}{2}\log{ ((Q_0+\sigma)^2+(|t|+T)^2 )}\\
&+\frac{c_2(1+\eta-\sigma)}{\eta}\log{\log{ ((Q_0+\sigma)^2+(|t|+T)^2 )}}.
\end{align*}

 For $\frac{1}{2} \leq \sigma \leq 1$, from
\begin{align*}
|f_N(s)| &\leq \left( k_1  ((Q_2 + \sigma)^2 + (|t|+T)^2 )^{\frac{k_2 +1}{2}} \left( \log \sqrt{(Q_2 + \sigma)^2 + (|t|+T)^2}\right)^{k_3}\right)^{(2-2 \sigma)N}\\
&\times
\left( c_1   ((Q_2 + \sigma)^2 + (|t|+T)^2  )^{\frac{1}{2}} \left( \log \sqrt{(Q_2 + \sigma)^2 + (|t|+T)^2}\right)^{c_2}\right)^{(2\sigma-1)N},
\end{align*}
it follows that 
\begin{align*}
\frac{1}{N} \log |f_N(s)| 
&\leq (2 - 2\sigma) \log k_1 + (2\sigma -1) \log c_1  - ( k_3 (2-2\sigma) + c_2 (2\sigma -1) )  \log 2\\
&+   \frac{(2-2\sigma)(k_2+1) + 2\sigma-1}{2}   \log ((Q_2 + \sigma)^2 + (|t|+T)^2 )\\
&+ ( k_3 (2-2\sigma) + c_2 (2\sigma -1) )\log \log  ((Q_2 + \sigma)^2 + (|t|+T)^2 ) .\\
\end{align*}

 Assume $0\leq \sigma\leq\frac{1}{2}$. From
\begin{align*}
|f_N(s)|&\leq ((\sigma-1)^2+(|t|+T)^2 )^{\frac{N}{2}}\\
&\times\left(\frac{c_1}{\sqrt{2\pi}}((Q_5+\sigma)^2+(|t|+T)^2)^{\frac{1}{4}}\left(\log{\sqrt{(Q_5+\sigma)^2+(|t|+T)^2}}\right)^{c_2}\right)^{(1-2\sigma)N}\\
&\times\left(k_1 ((Q_5+\sigma)^2+(|t|+T)^2 )^{\frac{k_2}{2}}
 \left(\log{\sqrt{(Q_5+\sigma)^2+(|t|+T)^2}}\right)^{k_3}\right)^{2\sigma N},
\end{align*}
 we derive
\begin{align*}
\frac{1}{N}\log{|f_N(s)|}&\leq (1-2\sigma)\log{\left(\frac{c_1}{2^{c_2+\frac{1}{2}}\sqrt{\pi}}\right)}+2\sigma\log{\frac{k_1}{2^{k_3}}}+\frac{1}{2}\log{ ((\sigma-1)^2+(|t|+T)^2 )}\\
&+\frac{1-2\sigma+4k_2\sigma}{4}\log{ ((Q_5+\sigma)^2+(|t|+T)^2 )}\\ 
&+(c_2(1-2\sigma)+2k_3\sigma)\log{\log{ ((Q_5+\sigma)^2+(|t|+T)^2 )}}.
\end{align*}

 For $-\eta \leq \sigma \leq 0$,  we have
\begin{align*}
|f_N(s)| &\leq  \left(\sqrt{ (\sigma-1)^2 + (|t|+T)^2} \right)^N \left( \frac{1}{(2\pi)^{\frac{1}{2} +  \eta}} \zeta(1+\eta) \left(\sqrt{(Q_4 + \sigma)^2 + (|t|+T)^2} \right)^{\frac{1}{2} + \eta}\right)^{-\frac{\sigma}{\eta} N}\\
&\times \left( \frac{c_1}{\sqrt{2\pi}} \left(\sqrt{(Q_4 + \sigma)^2 + (|t|+T)^2} \right)^{\frac{1}{2}} \left( \log \sqrt{(Q_4 + \sigma)^2 + (|t|+T)^2} \right)^{c_2}\right)^{\frac{\sigma + \eta}{\eta} N}.
\end{align*}
Thus,
\begin{align*}
\frac{1}{N} \log |f_N(s)| &\leq  -\frac{\sigma}{\eta} \log \frac{1}{(2\pi)^{\frac{1}{2} + \eta}} - \frac{\sigma}{\eta} \log(1+\eta) + \frac{\sigma + \eta}{\eta} \log \frac{c_1}{\sqrt{2\pi}} - \frac{\sigma + \eta}{\eta} c_2 \log 2\\
&+ \frac{1}{2} \log  ( (\sigma-1)^2 + (|t|+T)^2 ) \\
&+\left( -\frac{\sigma( 1+2\eta)}{4\eta}+ \frac{\sigma+ \eta}{4 \eta}\right) \log  ((Q_4 + \sigma)^2 + (|t|+T)^2  )\\
&+ \frac{\sigma + \eta}{\eta} c_2 \log \log ( (Q_4 + \sigma)^2 + (|t|+T)^2  ).
\end{align*}

Lastly, for $\sigma\leq -\eta$, as
 \begin{align*}
|f_N(s)|
&\leq ((\sigma-1)^2+(|t|+T)^2 )^{\frac{N}{2}}\left(\frac{1}{2\pi}\right)^{N(\frac{1}{2}-\sigma)}((1+\sigma-[\sigma])^2+(|t|+T)^2  )^{\frac{(1-2\sigma+2[\sigma])N}{4}}\\
&\times\left(\prod_{j=1}^{-[\sigma]}((\sigma+j-1)^2+(|t|+T)^2)\right)^{\frac{N}{2}}\zeta(1-\sigma)^{N},
\end{align*}
we obtain
\begin{align*}
\frac{1}{N}\log|f_N(s)|&\leq \log\zeta(1-\sigma)+\frac{1}{2}\log((\sigma-1)^2+(|t|+T)^2)\\
&+\frac{2\sigma-1}{2}\log{2\pi}+\frac{(1-2\sigma+2[\sigma])}{4}\log((1+\sigma-[\sigma])^2+(|t|+T)^2)\\
&+\frac{1}{2}\sum_{j=1}^{-[\sigma]} \log((\sigma+j-1)^2+(|t|+T)^2).
\end{align*}

\subsubsection{Constructing and bounding $F_{c,r} (\theta)$}\label{bd-F}

With the above bounds of $\frac{1}{N}\log|f_N(s)|$ in mind, similar to the construction of $F_{c,r}(\theta)$ for Dirichlet $L$-functions in \cite[Definition 5.10]{BMOR20} and Dedekind zeta functions in \cite{HSW}, we shall construct $F_{c,r}(\theta)$ for $\zeta(s)$ as follows. (We note that the main difference between our construction and the ones in \cite{BMOR20,HSW} lies in the range $0\le \Re(s)\le 1+\eta$ as we have sharper bounds for $\zeta(s)$ in this range.)

Similar to \cite{BMOR20,HSW}, we first introduce some auxiliary functions and notation. For $\theta\in[-\pi,\pi]$, we let  $\sigma=c+r\cos\theta$, with $c-r>-\frac{1}{2}$, and $t=r\sin\theta$. We define \begin{align*}
&L_j(\theta) =\log\frac{(j+c+r\cos\theta)^2+(|r\sin\theta|+T)^2}{T^2},\\
&M_j (\theta) = \log \log ( (j+c+r\cos\theta)^2+(|r\sin\theta|+T)^2 ) - \log\log (T^2).
\end{align*}
Now, we give  upper bounds for $L_j(\theta)$ and $M_j (\theta)$. From the inequality $\log x\leq x-1$, it follows that 
\begin{align*}
L_j(\theta)
\leq \frac{(j+c+r\cos\theta)^2+(|r\sin\theta|+T)^2}{T^2}  -1  = \frac{(j+c+r\cos\theta)^2+(r\sin\theta)^2}{T^2}  + \frac{2r\sin\theta}{T}
\end{align*}
for $\theta\in [0,\pi]$.
Fix $T_0\ge 1$. For  $\theta\in [0,\pi]$,  we let
\[
L^{\star}_j (\theta ) =\frac{1}{T_0}(j+c+r\cos\theta)^2+ \frac{1}{T_0}(r\sin\theta)^2 + 2r \sin \theta.
\]
It is clear that for $T\ge T_0$ and $\theta\in [0,\pi]$,
\[
L_j(\theta) \leq \frac{L^{\star}_j (\theta ) }{T}.
\]
Similarly, for $T\ge T_0$ and $\theta\in [0,\pi]$, we have
\begin{align*}
M_j (\theta) &\leq \frac{\log ( (j+c+r\cos\theta)^2+(|r\sin\theta|+T)^2  )}{\log  (T^2)} - 1\\
&= \frac{1}{\log (T^2)} \left( \log \frac{ (j+c+r\cos\theta)^2+(|r\sin\theta|+T)^2 }{T^2} + \log (T^2)\right) - 1\\
&=  \frac{1}{\log (T^2)}L_j(\theta) \\
&\leq \frac{L^{\star}_j (\theta ) }{2T\log T}.
\end{align*}

We are now in a position to construct $F_{c,r}(\theta)$.

\begin{defn}
\label{Defn5.10}
For $\theta\in[-\pi,\pi]$, we let  $\sigma=c+r\cos\theta$, with $c-r>-\frac{1}{2}$, and $t=r\sin\theta$. For $\sigma > 1+ \eta$, we define
\begin{align*}
F_{c,r} (\theta) = \frac{1}{2}L_{-1}(\theta) + \log T + \log \zeta(\sigma).
\end{align*}
For $1\leq \sigma \leq 1+\eta$, we define
\begin{align*}
F_{c,r}(\theta)&=\frac{1+\eta-\sigma}{\eta}\log{c_1}+\frac{\sigma-1}{\eta}\log{\zeta(1+\eta)}+\frac{1}{2}L_{Q_0}(\theta)+\log{T}  \\
&+\frac{c_2(1+\eta-\sigma)}{\eta}M_{Q_0}(\theta)+\frac{c_2(1+\eta-\sigma)}{\eta}\log{\log{T}}.
\end{align*}
For $\frac{1}{2} \leq \sigma \leq 1$, we define
\begin{align*}
F_{c,r} (\theta)
&=  (2 - 2\sigma) \log k_1 + (2\sigma -1) \log c_1 +  ((2-2\sigma)(k_2+1) + 2\sigma-1 )\left( \frac{L_{Q_2}(\theta)}{2} + \log T\right)\\
&+( k_3 (2-2\sigma) + c_2 (2\sigma -1))( M_{Q_2}(\theta) + \log \log T) .
\end{align*}
For $0\leq \sigma\leq\frac{1}{2}$, we define
\begin{align*}
F_{c,r}(\theta)&=(1-2\sigma)\log{ \frac{c_1}{\sqrt{2\pi}} }+2\sigma\log{k_1}+\frac{1}{2}L_{-1}(\theta)+\log{T}+\frac{1-2\sigma+4k_2\sigma}{2}\left(\frac{L_{Q_5}(\theta)}{2}+\log{T}\right)\\ 
&+(c_2(1-2\sigma)+2k_3\sigma)(M_{Q_5}(\theta)+ \log{\log{T}}).
\end{align*}
For $-\eta \leq \sigma \leq 0$, we define
\begin{align*}
F_{c,r} (\theta) &=  -\frac{\sigma}{\eta} \log \frac{1+\eta}{c_1(2\pi)^{ \eta}} + \log \frac{c_1}{\sqrt{2\pi}} + \frac{1}{2} L_{-1} (\theta) + \log T\\
&+\left( -\frac{\sigma( 1+2\eta)}{2\eta}+ \frac{\sigma+ \eta}{2 \eta}\right)\left(\frac{L_{Q_4}(\theta)}{2} +\log T\right)  + \frac{\sigma+ \eta}{\eta} c_2 ( M_{Q_4}(\theta) +\log \log T).
\end{align*}
For $\sigma\leq -\eta$, we define
\begin{align*}
F_{c,r}(\theta)&=\log\zeta(1-\sigma)+\frac{1}{2}L_{-1}(\theta)+\left(1+ \frac{1-2\sigma}{2}\right)\log{T}-\frac{1-2\sigma}{2}\log 2\pi\\ 
&+\frac{(1-2\sigma+2[\sigma])}{4}L_{1-[\sigma]}(\theta)+ \frac{1}{2}\sum_{j=1}^{-[\sigma]} L_{j-1}(\theta).
\end{align*}

\end{defn}
It is clear that $F_{c,r}(\theta)$ is an even function of $\theta$ such that $F_{c,r}(\theta)\geq\frac{1}{N}\log|f_{N}(c+re^{i\theta})|$.

Now, we shall provide an upper bound for $\int_{0}^\pi F_{c,r} (\theta) d\theta$. In light of work of \cite{BMOR20} and \cite{HSW}, for $c\in\mathbb{R}$ and $r>0$, we define 
\begin{equation}\label{def-theta}
\theta_{y}=
\begin{cases}
0 & \text{if $c+r\leq y$;} \\
\arccos\frac{y-c}{r} & \text{if $c-r\leq y\leq c+r$;} \\
\pi & \text{if $y \leq c-r$.}
\end{cases}
\end{equation}

Now, we let $c,r$, and $\eta$ be positive real numbers satisfying 
\begin{equation}\label{long-cond-1}
-\frac{1}{2}< c-r <1-c  <-\eta <1+\eta < c
\end{equation}
and $0<\eta\leq\frac{1}{2}$. To bound $\int_0^{\pi} F_{c,r}(\theta) d\theta$, we consider the splitting
\begin{align*}
\int_0^{\pi}  = \int_0^{\theta_{1+\eta}} + \int_{\theta_{1+\eta}}^{\theta_1}  + \int_{\theta_1}^{\theta_{\frac{1}{2}}} 
+\int_{\theta_{\frac{1}{2}}}^{\theta_0} + \int_{\theta_0}^{\theta_{-\eta}} + \int_{\theta_{-\eta}}^{\pi}. 
\end{align*}

Firstly, we have
\begin{align*} 
\int_0^{\theta_{1+\eta} }F_{c,r} (\theta) d \theta 
\leq \log T \int_0^{\theta_{1+\eta} } 1 d \theta  +  \int_0^{\theta_{1+\eta} }   \log \zeta(\sigma) d \theta +  \frac{1}{2T}\int_0^{\theta_{1+\eta} }   L_{-1}^{\star}(\theta) d \theta.
\end{align*}

Secondly, $\int_{\theta_{1+\eta}}^{\theta_1} F_{c,r}(\theta) d\theta$ is bounded above by
\begin{align*}
&\log{T}  \int_{\theta_{1+\eta}}^{\theta_1 } 1 d \theta + \frac{c_2}{\eta}\log{\log{T}}\int_{\theta_{1+\eta}}^{\theta_1} (1+\eta-\sigma) d\theta + \frac{\log{c_1}}{\eta}\int_{\theta_{1+\eta}}^{\theta_1} (1+\eta-\sigma) d\theta \\
&+\frac{\log{\zeta(1+\eta)}}{\eta} \int_{\theta_{1+\eta}}^{\theta_1} (\sigma-1) d\theta + \frac{1}{2T}\int_{\theta_{1+\eta}}^{\theta_1} L^{\star}_{Q_0}(\theta) d\theta
+\frac{c_2}{2\eta T\log{T}}\int_{\theta_{1+\eta}}^{\theta_1} (1+\eta-\sigma)L^{\star}_{Q_0}(\theta) d\theta.
\end{align*}

A direction calculation shows that
\begin{align*}
\int_{\theta_{1} }^{\theta_{\frac{1}{2}}}F_{c,r} (\theta) d \theta  
 &\le \log T \int_{\theta_{1} }^{\theta_{\frac{1}{2}}}  (2-2\sigma)(k_2+1) + 2\sigma-1  d\theta
 +    \log\log T \int_{\theta_{1} }^{\theta_{\frac{1}{2}}}   k_3 (2-2\sigma) + c_2 (2\sigma -1)  d \theta\\
&+
\int_{\theta_{1} }^{\theta_{\frac{1}{2}}} (2 - 2\sigma) \log k_1 + (2\sigma -1) \log c_1 d\theta\\
 &+ 
 \frac{1}{T}\int_{\theta_{1} }^{\theta_{\frac{1}{2}}}
 \frac{(2-2\sigma)(k_2+1) + 2\sigma-1}{2}  L_{Q_2}^{\star}(\theta)d \theta\\
&+ \frac{1}{2 T \log T}\int_{\theta_{1} }^{\theta_{\frac{1}{2}}}  ( k_3 (2-2\sigma) + c_2 (2\sigma -1)  )  L_{Q_2}^{\star}(\theta) d \theta.
\end{align*}

Also, we can bound $\int_{\theta_{\frac{1}{2}}}^{\theta_0} F_{c,r}(\theta) d\theta$ above by
\begin{align*}
&\log{T} \int_{\theta_{\frac{1}{2}}}^{\theta_0}1 d\theta  +  \frac{\log{T}}{2}\int_{\theta_{\frac{1}{2}}}^{\theta_0} 1-2\sigma+4k_2\sigma  d\theta  + \log{\log{T}}\int_{\theta_{\frac{1}{2}}}^{\theta_0} c_2(1-2\sigma)+2k_3\sigma d\theta\\
& +\left( \log{\frac{c_1}{\sqrt{2\pi}}}\right) \int_{\theta_{\frac{1}{2}}}^{\theta_0}  1-2\sigma  d\theta+ 2\log{k_1}\int_{\theta_{\frac{1}{2}}}^{\theta_0} \sigma d\theta
+\frac{1}{2T}\int_{\theta_{\frac{1}{2}}}^{\theta_0} L^{\star}_{-1}(\theta) d\theta\\
&+\frac{1}{4T}\int_{\theta_{\frac{1}{2}}}^{\theta_0}(1-2\sigma+4k_2\sigma)L^{\star}_{Q_5}(\theta) d\theta
+ \frac{1}{2T\log{T}}\int_{\theta_{\frac{1}{2}}}^{\theta_0}(c_2(1-2\sigma)+2k_3\sigma)L^{\star}_{Q_5}(\theta) d\theta .
\end{align*}

Moreover, $\int_{\theta_{0} }^{\theta_{{-\eta}}}F_{c,r} (\theta) d \theta$ is bounded above by  
\begin{align*}
&\log T \int_{\theta_{0} }^{\theta_{-\eta}} 1 - \frac{\sigma( 1+2\eta)}{2\eta}+ \frac{\sigma+ \eta}{2 \eta}  d\theta+  \log \log T \int_{\theta_{0} }^{\theta_{{-\eta}}} \frac{\sigma+ \eta}{\eta} c_2 d \theta   \\
&+ \int_{\theta_{0} }^{\theta_{{-\eta}}}
-\frac{\sigma}{\eta} \log \frac{1+\eta}{c_1(2\pi)^{ \eta}} + \log \frac{c_1}{\sqrt{2\pi}} 
d\theta +\frac{1}{2T}\int_{\theta_{0} }^{\theta_{-\eta}}L_{-1}^{\star}( \theta) d\theta \\
& + \frac{1}{T}\int_{\theta_{0} }^{\theta_{{-\eta}}}
 \left( -\frac{\sigma( 1+2\eta)}{4\eta}+ \frac{\sigma+ \eta}{4 \eta}\right) L_{Q_4}^{\star}(\theta)d \theta + \frac{1}{2 T \log T}\int_{\theta_{0} }^{\theta_{{-\eta}}}\frac{\sigma+ \eta}{\eta} c_2 L_{Q_4}^{\star}(\theta) d \theta.
\end{align*}

Finally, we have
\begin{align*}
\int_{\theta_{-\eta}}^{\pi}F_{c,r}(\theta)d\theta 
&\leq  \log{T} \int_{\theta_{-\eta}}^{\pi}1+ \frac{1-2\sigma}{2}d\theta+ \int_{\theta_{-\eta}}^{\pi}\log{\zeta(1-\sigma)}d\theta- \log 2\pi \int_{\theta_{-\eta}}^{\pi}\frac{1-2\sigma}{2}d\theta\\
&+\frac{1}{2T}\int_{\theta_{-\eta}}^{\pi}L^{\star}_{-1}(\theta)d\theta +\frac{1}{T}\int_{\theta_{-\eta}}^{\theta_{-\frac{1}{2}}}\frac{1-2\sigma}{4}L^{\star}_{1}(\theta)d\theta\\ 
&+\sum_{j=1}^{\infty} \int_{\theta_{-j+\frac{1}{2}}}^{\theta_{-j-\frac{1}{2}}}\left( \frac{1-2\sigma-2j}{4}L_{j+1}(\theta)+\frac{1}{2}\sum_{k=1}^j L_{k-1}(\theta)\right) d\theta.
\end{align*}
(We note that from the assumption $-\frac{1}{2}< c-r$ it follows that $\theta_{-j+\frac{1}{2}} = \theta_{-j-\frac{1}{2}} = \pi$ for $j \geq 1$, and thus the last term in the above inequality is equal to zero.)

To end this section, we require the following two estimates from  \cite[Lemmata 5.14 and 5.15]{BMOR20} to control the zeta integrals appearing in the above estimates.

\begin{lemma}\label{zeta-int}
Let $c,r$ and $\eta$ be positive real numbers, satisfying \eqref{long-cond-1}, and $J_1$ and $J_2$ be positive integers.  If $\theta_{1 + \eta} \leq 2.1$, then for $\sigma = c+ r\cos \theta$, one has
\begin{align*}
\int_0^{\theta_{1 + \eta}} \log \zeta (\sigma) d\theta \leq \frac{\log \zeta (1 + \eta) + \log \zeta (c)}{2} \left(\theta_{1 + \eta} - \frac{\pi}{2} \right) + \frac{\pi}{4 J_1} \log \zeta (c) + \kappa_1 (J_1),
\end{align*}
where
\begin{align*}
\kappa_1(J_1)=\frac{\pi}{4J_1}\left(\log\zeta(c+r)+2\sum_{j=1}^{J_1-1} \log\zeta \left(c+r\cos\frac{\pi j}{2J_1}\right)\right).
\end{align*}
In addition, assuming further $r > 2c  -1$, one has
\begin{align*}
\int_{\theta_{-\eta}}^\pi \log \zeta(1 - \sigma) d \theta \leq \frac{\log \zeta (1 + \eta) + \log \zeta (c)}{2} (\theta_{1-c} - \theta_{-\eta}) + \frac{\pi - \theta_{1-c}}{2 J_2} \log \zeta (c) + \kappa_2 (J_2),
\end{align*}
where 
\begin{align*}
\kappa_2(J_2)=\frac{\pi-\theta_{1-c}}{2J_2}\left(\log\zeta(1-c+r)+2\sum_{j=1}^{J_2-1} \log\zeta\left(1-c-r\cos\left(\frac{\pi j}{J_2}+\left(1-\frac{j}{J_2}\right)\theta_{1-c}\right)\right)\right).
\end{align*}
\end{lemma}

\section{Final formulae}\label{final}

In this section, we shall first prove Theorem \ref{main-thm-0}.
\begin{proof}[Proof of Theorem \ref{main-thm-0}]
 Using \eqref{main-est} and Propositions \ref{Je-bd} and \ref{BT}, for
$$
-\frac{1}{2}<c-r<1-c < -\eta <0 <\frac{1}{4} \le \delta=2c- \sigma_1 -\frac{1}{2} < \frac{1}{2} < 1 < 1+\eta < c< \sigma_1 =c+ \frac{(c-1/2)^2}{r}< c+r,
$$
satisfying $\theta_{1 + \eta} \leq 2.1$,   we have
\begin{align*}
 &\left|  N_{\Bbb{Q}}(T)  -  \frac{T}{\pi} \log\left( \frac{T}{2\pi e}\right) + \frac{1}{4}\right| \\
&\leq  
 \frac{5}{2} +\frac{1}{25 T}+ \frac{2}{\pi} \log \zeta(\sigma_1)+\frac{1}{ \log (r/ (c-1/2))} \log\frac{\zeta(c)}{\zeta(2c)}  -\frac{1}{ \log (r/ (c-1/2))} \log T\\
 &+ \frac{1}{\pi \log (r/ (c-1/2))} \int_{0}^{\pi}F_{c,r}(\theta)d\theta + \frac{E(T,\delta)}{\pi}.
\end{align*}
Thus, recalling that $N_{\Bbb{Q}}(T)=2N(T)$ and applying Lemma \ref{EK-bd-final} (to bound $E(T,\delta)$) and the estimates from Section \ref{bd-F} (to bound $\int_{0}^{\pi}F_{c,r}(\theta)d\theta$), for $T \geq T_0$, we have
\begin{align*}
 \left| N(T)   -  \frac{T}{ 2 \pi} \log\left( \frac{T}{2\pi e}\right) + \frac{1}{ 8} \right| 
  \leq C_1 \log T + C_2 \log \log T + C_3,
\end{align*}
where for $j=1,2,3$,
\begin{align}\label{def-Cj}
C_j = \frac{\tilde{C}_j}{ 2 \pi \log (r/ (c-1/2))},
\end{align}
\begin{align}\label{def-C1}
 \begin{split}
\tilde{C}_1&= 
\int_{\theta_{1} }^{\theta_{\frac{1}{2}}} (2-2\sigma)(k_2+1) + 2\sigma-2    d\theta+ \frac{1}{2}\int_{\theta_{\frac{1}{2}}}^{\theta_0} 1-2\sigma+4k_2\sigma  d\theta\\
&+ \int_{\theta_{0} }^{ \theta_{-\eta}}  -  \frac{ \sigma( 1+2\eta)}{2\eta}+ \frac{\sigma+ \eta}{2 \eta}  d\theta
+\int_{\theta_{-\eta}}^{\pi}\frac{1-2\sigma}{2}d\theta ,
 \end{split}
\end{align}
\begin{align}\label{def-C2}
  \begin{split}
\tilde{C}_2 & =\frac{c_2}{\eta}\int_{\theta_{1+\eta}}^{\theta_1} (1+\eta-\sigma) d\theta
+\int_{\theta_{1} }^{\theta_{\frac{1}{2}}}   k_3 (2-2\sigma) + c_2 (2\sigma -1)  d \theta\\
&+\int_{\theta_{\frac{1}{2}}}^{\theta_0} c_2(1-2\sigma)+2k_3\sigma d\theta +\int_{\theta_{0} }^{\theta_{{-\eta}}} \frac{\sigma+ \eta}{\eta} c_2 d \theta,
 \end{split}
\end{align}
\begin{align}\label{def-C3}
 \begin{split}
\tilde{C}_3 &= 
 \pi \log (r/ (c-1/2))\left(
\frac{640\delta-112}{1536(3T_0-1)} + \frac{1}{2^{10}}+\frac{5}{2} +\frac{1}{25 T_0}+ \frac{2}{\pi} \log \zeta(\sigma_1)\right)+  \pi\log\frac{\zeta(c)}{\zeta(2c)}\\
&+ \frac{\log{c_1}}{\eta}\int_{\theta_{1+\eta}}^{\theta_1} (1+\eta-\sigma) d\theta +\frac{\log{\zeta(1+\eta)}}{\eta} \int_{\theta_{1+\eta}}^{\theta_1} (\sigma-1) d\theta  \\
&+\int_{\theta_{1} }^{\theta_{\frac{1}{2}}} (2 - 2\sigma) \log k_1 
+ (2\sigma -1) \log c_1 d\theta
+\left(\log{\frac{c_1}{\sqrt{2\pi}}}\right) \int_{\theta_{\frac{1}{2}}}^{\theta_0} 1-2\sigma d\theta + 2\log{k_1}\int_{\theta_{\frac{1}{2}}}^{\theta_0} \sigma d\theta\\
&+\int_{\theta_{0} }^{\theta_{{-\eta}}}
-\frac{\sigma}{\eta} \log \frac{1+\eta}{c_1(2\pi)^{ \eta}}+ \log \frac{c_1}{\sqrt{2\pi}} 
d\theta 
-(\log 2\pi) \int_{\theta_{-\eta}}^{\pi}\frac{1-2\sigma}{2}d\theta \\
&+  \frac{\log \zeta (1 + \eta) + \log \zeta (c)}{2} \left(\theta_{1 + \eta} - \frac{\pi}{2} \right) + \frac{\pi}{4 J_1} \log \zeta (c) + \kappa_1 (J_1) \\
 &+ \frac{\log \zeta (1 + \eta) + \log \zeta (c)}{2} (\theta_{1-c} - \theta_{-\eta}) + \frac{\pi - \theta_{1-c}}{2 J_2} \log \zeta (c) + \kappa_2 (J_2) + \kappa_3(T_0),
 \end{split}
\end{align}
and $\kappa_3(T_0)$ is equal to
\begin{align*}
&\frac{1}{2T_0}\max \Big\{0,\int_0^{\theta_{1+\eta} }   L_{-1}^{\star}(\theta) +\int_{\theta_{1+\eta}}^{\theta_1} L^{\star}_{Q_0}(\theta) d\theta+ \int_{\theta_{1} }^{\theta_{\frac{1}{2}}}
  ({(2-2\sigma)(k_2+1) + 2\sigma-1} ) L_{Q_2}^{\star}(\theta)d \theta\\
& +\int_{\theta_{\frac{1}{2}}}^{\theta_0} L^{\star}_{-1}(\theta) d\theta 
+\frac{1}{2}\int_{\theta_{\frac{1}{2}}}^{\theta_0}(1-2\sigma+4k_2\sigma)L^{\star}_{Q_5}(\theta) d\theta+  \int_{\theta_{0} }^{\theta_{{-\eta}}} L_{-1}^{\star} (\theta) d \theta  \\
 &+ 
\int_{\theta_{0} }^{\theta_{{-\eta}}}
 \left( - \frac{\sigma( 1+2\eta)}{2\eta}+ \frac{\sigma+ \eta}{2 \eta}\right) L_{Q_4}^{\star}(\theta)d \theta +\int_{\theta_{-\eta}}^{\pi}L^{\star}_{-1}(\theta)d\theta
 +\int_{\theta_{-\eta}}^{\theta_{-\frac{1}{2}}}\frac{1-2\sigma}{2}L^{\star}_{1}(\theta)d\theta 
\Big\}\\
&+\frac{1}{2T_0\log{T_0}} \max \Big\{ 0, \int_{\theta_{1+\eta}}^{\theta_1} \frac{c_2}{\eta}(1+\eta-\sigma)L^{\star}_{Q_0}(\theta) d\theta + \int_{\theta_{1} }^{\theta_{\frac{1}{2}}} ( k_3 (2-2\sigma) + c_2 (2\sigma -1) )  L_{Q_2}^{\star}(\theta) d \theta\\
&+\int_{\theta_{\frac{1}{2}}}^{\theta_0}(c_2(1-2\sigma)+2k_3\sigma)L^{\star}_{Q_5}(\theta) d\theta
+\int_{\theta_{0} }^{\theta_{{-\eta}}}\frac{\sigma+ \eta}{\eta} c_2 L_{Q_4}^{\star}(\theta) d \theta 
\Big\}.
\end{align*}

Recall that Patel \cite{Pa20-1} showed that $|\zeta(1+it)|\le \log t$ for $t\ge 3$ (i.e., in \eqref{cond-1}, $(c_1,c_2,t_0)=(1, 1, 3)$ is admissible) and that by the work of Hiary \cite{Hi15},   $(k_1,k_2,k_3, t_1) = (0.77, \frac{1}{6}, 1, 3)$ is admissible for \eqref{cond-2}.\footnote{In fact, Hiary \cite{Hi15} showed that  $(k_1,k_2,k_3) = (0.63, \frac{1}{6}, 1)$  was admissible. However, as pointed out by \cite{Pa20-1}, due to an error in \cite{Hi15}, one can only take $(k_1,k_2,k_3) = (0.77, \frac{1}{6}, 1)$.

We also note that there are two bounds used in \cite[Sec. 5]{Tr14-2} and \cite{PT} that may be no longer valid. On one hand, Trudgian in \cite[Sec. 5]{Tr14-2} used a result from \cite{Tr14-1} that  $|\zeta(1+it)|\le \frac{3}{4} \log t$ for $t\ge 3$. On the other hand,  \cite[Theorem 1]{PT} states that $(k_1,k_2,k_3) = (0.732, \frac{1}{6}, 1)$ is admissible. However, as pointed out by \cite{Pa20-1}, both bounds made use of an incorrect result obtained by  Cheng-Graham in \cite{CG04}. We refer the reader to \cite{Pa20-1} for a detailed discussion.} In addition, for $(c_1,c_2,t_0)=(1, 1, 3)$ and $(k_1,k_2,k_3,t_1) = (0.77, \frac{1}{6}, 1, 3 )$, we may take
$$
(Q_0,Q_1,Q_2,Q_3,Q_4,Q_5)=(1, 1.18, 1.18, 3.9, 2.3, 3.9).
$$
Finally, for $T_0= 30\,610\,046\,000$, choosing $J_1=64$ and $J_2=39$, we calculate admissible $(C_1,C_2,C_3)$ and record them in Table \ref{table2}.
\end{proof}

\begin{table}[htbp] 
\centering
\begin{tabular}{ |c|c|c||c|c|c|c|   } 
 \hline  
 $c$ & $r$ & $\eta$ &  $C_1$ & $C_2$ & $C_3$ & $C'_3$  \\ 
 \hline
 1.000011314 & 1.064340602 & $4.2826451\cdot 10^{-6}$   & 0.103787  &  0.257297  & 9.367419 &  8.367419     \\ 
 \hline
 1.025253504 & 1.182375395 & 0.009944751381   & 0.109410 &  0.204142 & 4.030486 &  3.030486  \\ 
 \hline 
 1.035766557 & 1.229059659 & 0.014325507360   & 0.111973   & 0.189768 & 3.746756 & 2.746756\\ 
 \hline 
\end{tabular}
   \caption{Choices of parameters $(c,r,\eta)$ and resulting admissible $(C_1,C_2,C_3, C'_3)$}\label{table2} 
\end{table}

\newpage

Now, we are in a position to prove Corollary \ref{main-thm}.

\begin{proof}[Proof of Corollary \ref{main-thm}]
By Theorem \ref{main-thm-0} (with $T_0 = 30\,610\,046\,000$ and Table \ref{table2}), it is sufficient to verify the corollary for $e\le T \leq 30\,610\,046\,000$. We note that by \eqref{bd-N-K-1}, 
\begin{align}\label{formula-S}
S(T)= \frac{1}{\pi} \Delta_{\mathcal{C}_0} \arg \zeta(s) = \frac{1}{2}\left(  N_{\Bbb{Q}}(T) - \frac{T}{\pi} \log \left(\frac{T}{2\pi e}\right) + \frac{1}{4} - g(T) -2\right), 
 \end{align}
where  $ g(T)$ is defined as in \eqref{def-g}. Thus, by \eqref{bd-S} and \eqref{bd-g}, we obtain
\begin{equation}\label{for-referee}
\left| N(T) - \frac{T}{2\pi} \log \left(\frac{T}{2\pi e}\right) + \frac{1}{8}\right|
\le |S(T)|+ \frac{1}{2}|g(T)| +1 
\le  2.5167  +\frac{1}{50e} +1 
\end{equation}
for $e\le T \leq 30\,610\,046\,000$. (We remark that one may apply \cite[Lemma 2]{BPT} to improve \eqref{for-referee} for $T\ge 2\pi$.) As the quantity on the right of \eqref{for-referee} clearly is less than   $0.1038  \log T + 0.2573  \log\log T + 9.3675$ for any $T\ge e$. We then conclude the proof by using the triangle inequality.
\end{proof}

To end this section, we shall prove Theorem \ref{main-thm-S(T)}.

\begin{proof}[Proof of Theorem \ref{main-thm-S(T)}]
Note that
$$
S(T)= \frac{1}{\pi} \Delta_{\mathcal{C}_0} \arg \zeta(s) = \frac{1}{\pi} \Delta_{\mathcal{C}_1} \arg \zeta(s) + \frac{1}{\pi} \Delta_{\mathcal{C}_2} \arg (s-1)\zeta(s) -  \frac{1}{\pi} \Delta_{\mathcal{C}_2} \arg (s-1). 
$$
We know that
$$
|\Delta_{\mathcal{C}_1} \arg \zeta(s)|\le \log \zeta(\sigma_1)
$$
and
$$
 |\Delta_{\mathcal{C}_2} \arg (s-1)|= \arctan\left(\frac{\sigma_1 -1}{ T}\right) +\arctan\left(\frac{1}{ 2T}\right)\le \arctan\left(\frac{\sigma_1 -1}{ T_0}\right) +\arctan\left(\frac{1}{ 2T_0}\right)
$$ 
for $T\ge T_0$. 
From Propositions \ref{Je-bd} and \ref{BT}, it follows that for
$$
-\frac{1}{2}<c-r<1-c < -\eta <0 <\frac{1}{4} \le \delta=2c- \sigma_1 -\frac{1}{2} < \frac{1}{2} < 1 < 1+\eta < c< \sigma_1 =c+ \frac{(c-1/2)^2}{r}< c+r,
$$
satisfying $\theta_{1 + \eta} \leq 2.1$,   
\begin{align*}
|S(T)|&\leq  
 \frac{1}{4} + \frac{1}{\pi} \log \zeta(\sigma_1)+\frac{1}{2 \log (r/ (c-1/2))} \log\frac{\zeta(c)}{\zeta(2c)}  -\frac{1}{2 \log (r/ (c-1/2))} \log T\\
 &+ \frac{1}{2\pi \log (r/ (c-1/2))} \int_{0}^{\pi}F_{c,r}(\theta)d\theta + \frac{E(T,\delta)}{ 2\pi} + \frac{1}{\pi}\arctan\left(\frac{\sigma_1 -1}{ T_0}\right) + \frac{1}{\pi} \arctan\left(\frac{1}{ 2T_0}\right).
\end{align*} 
Thus, applying Lemma \ref{EK-bd-final} (to bound $E(T,\delta)$) and the estimates from Section \ref{bd-F} (to bound $\int_{0}^{\pi}F_{c,r}(\theta)d\theta$), for $T \geq T_0$, we have
\begin{align*}
|S(T)|
  \leq C_1 \log T + C_2 \log \log T + C_3',
\end{align*}
where
\begin{equation}\label{def-D3}
C_3' = C_3 -1  +\frac{1}{\pi} \arctan\left(\frac{\sigma_1 -1}{ T_0}\right) + \frac{1}{\pi}\arctan\left(\frac{1}{ 2T_0}\right),
\end{equation}
$C_1=C_1(c,r,\eta;k_2)$, $C_2=C_2(c,r,\eta;c_2,k_3) $, $C_3=C_3(c,r,\eta;c_1,c_2,t_0 ,k_1,k_2,k_3,t_1;T_0)$ are given in \eqref{def-Cj}, \eqref{def-C1}, \eqref{def-C2}, and \eqref{def-C3}. In particular, for $T_0= 30\,610\,046\,000$, we have  admissible $(C_1,C_2,C'_3)$ recorded in Table \ref{table2}. Finally, applying \eqref{bd-S}, we conclude the proof.
\end{proof}

\section*{Acknowledgments}
The authors would like to thank Nathan Ng for his encouragement and helpful comments. They are also grateful to the referee for the constructive and detailed comments.

\begin{rezabib}

\bib{Ba18}{article}{
AUTHOR={Backlund, R. J.},
     TITLE = {\"{U}ber die {N}ullstellen der {R}iemannschen {Z}etafunktion},
   JOURNAL = {Acta Math.},
  FJOURNAL = {Acta Mathematica},
    VOLUME = {41},
      YEAR = {1916},
    NUMBER = {1},
     PAGES = {345--375},
      ISSN = {0001-5962},
   MRCLASS = {DML},
  MRNUMBER = {1555156},
       DOI = {10.1007/BF02422950},
       URL = {https://doi.org/10.1007/BF02422950},
}

\bib{BMOR20}{article}{
AUTHOR={Bennett, M. A.},
AUTHOR={Martin, G.},
AUTHOR={O'Bryant, K.},
AUTHOR={Rechnitzer, A.},
TITLE = {Counting zeros of {D}irichlet {$L$}-functions},
   JOURNAL = {Math. Comp.},
  FJOURNAL = {Mathematics of Computation},
    VOLUME = {90},
      YEAR = {2021},
    NUMBER = {329},
     PAGES = {1455--1482},
      ISSN = {0025-5718},
   MRCLASS = {11N13 (11M20 11M26 11N37 11Y35 11Y40)},
  MRNUMBER = {4232231},
       DOI = {10.1090/mcom/3599},
       URL = {https://doi.org/10.1090/mcom/3599},
}

\bib{BPT}{unpublished}{
AUTHOR={Brent, R. P.},
AUTHOR={Platt, D. J.},
AUTHOR={Trudgian, T. S.},
TITLE={Accurate estimation of sums over zeros of the Riemann zeta-function},
NOTE={arXiv:2009.13791},
JOURNAL={},
PAGES={},
VOLUME={},
NUMBER={},
YEAR={2020},
}

\bib{CG04}{article}{
AUTHOR={Cheng, Y. F.},
AUTHOR={Graham, S. W.},
     TITLE = {Explicit estimates for the {R}iemann zeta function},
   JOURNAL = {Rocky Mountain J. Math.},
  FJOURNAL = {The Rocky Mountain Journal of Mathematics},
    VOLUME = {34},
      YEAR = {2004},
    NUMBER = {4},
     PAGES = {1261--1280},
      ISSN = {0035-7596},
   MRCLASS = {11M06 (11L07 11Y35)},
  MRNUMBER = {2095256},
MRREVIEWER = {Haseo Ki},
       DOI = {10.1216/rmjm/1181069799},
       URL = {https://doi.org/10.1216/rmjm/1181069799},
}

\bib{FK15}{article}{
AUTHOR={Faber, L.},
AUTHOR={Kadiri, H.},
     TITLE = {New bounds for {$\psi(x)$}},
   JOURNAL = {Math. Comp.},
  FJOURNAL = {Mathematics of Computation},
    VOLUME = {84},
      YEAR = {2015},
    NUMBER = {293},
     PAGES = {1339--1357},
      ISSN = {0025-5718},
   MRCLASS = {11M06 (11M26)},
  MRNUMBER = {3315511},
MRREVIEWER = {Yuanyou Furui Cheng},
       DOI = {10.1090/S0025-5718-2014-02886-X},
       URL = {https://doi.org/10.1090/S0025-5718-2014-02886-X},
}

\bib{Gr13}{article}{
AUTHOR={Grossmann, J.},
TITLE={\"Uber die Nullstellen der Riemannschen Zeta-Funktion und der Dirichletschen $L$-Funktionen},
note={PhD thesis, Georg-August-Universit\"at G\"ottingen, 1913},
}

\bib{HSW}{article}{
AUTHOR={Hasanalizade, E.},
AUTHOR={Shen, Q.},
AUTHOR={Wong, P.-J.},
TITLE={Counting zeros of {D}edekind zeta functions},
NOTE={},
JOURNAL={accepted for publication in Mathematics of Computation},
PAGES={},
VOLUME={},
NUMBER={},
YEAR={},
}

\bib{Hi15}{article}{
    AUTHOR = {Hiary, G. A.},
     TITLE = {An explicit van der {C}orput estimate for {$\zeta(1/2+it)$}},
   JOURNAL = {Indag. Math. (N.S.)},
  FJOURNAL = {Koninklijke Nederlandse Akademie van Wetenschappen.
              Indagationes Mathematicae. New Series},
    VOLUME = {27},
      YEAR = {2016},
    NUMBER = {2},
     PAGES = {524--533},
      ISSN = {0019-3577},
   MRCLASS = {11M06 (11M26)},
  MRNUMBER = {3479170},
MRREVIEWER = {Frank Henry Thorne},
       DOI = {10.1016/j.indag.2015.10.011},
       URL = {https://doi.org/10.1016/j.indag.2015.10.011},
}

\bib{LMFDB}{article}{
 author       = {The {LMFDB Collaboration}},
  title        = {The {L}-functions and Modular Forms Database},
  note         = {\url{http://www.lmfdb.org} [Online; accessed 28 January 2021]},
}

\bib{Pa20-1}{unpublished}{
AUTHOR={Patel, D.},
TITLE={An explicit upper bound for $|\zeta(1+it)|$},
NOTE={arXiv:2009.00769},
JOURNAL={},
PAGES={},
VOLUME={},
NUMBER={},
YEAR={2020},
}

\bib{PT}{article}{
AUTHOR = {Platt, D. J.},
AUTHOR = {Trudgian, T. S.},
     TITLE = {An improved explicit bound on {$|\zeta(\frac 12+it)|$}},
   JOURNAL = {J. Number Theory},
  FJOURNAL = {Journal of Number Theory},
    VOLUME = {147},
      YEAR = {2015},
     PAGES = {842--851},
      ISSN = {0022-314X},
   MRCLASS = {11M06},
  MRNUMBER = {3276357},
MRREVIEWER = {Peter Jaehyun Cho},
       DOI = {10.1016/j.jnt.2014.08.019},
       URL = {https://doi.org/10.1016/j.jnt.2014.08.019},
}

\bib{PT21}{article}{
AUTHOR = {Platt, D. J.},
AUTHOR = {Trudgian, T. S.},
TITLE={The {R}iemann hypothesis is true up to $3 \cdot 10^{12}$},
NOTE={},
JOURNAL={Bull. London Math. Soc.},
PAGES={published online},
VOLUME={},
NUMBER={},
YEAR={2021},
DOI = {10.1112/blms.12460},
}

\bib{Ra59}{article}{
    AUTHOR = {Rademacher, H.},
     TITLE = {On the {P}hragm\'{e}n-{L}indel\"{o}f theorem and some applications},
   JOURNAL = {Math. Z},
  FJOURNAL = {Mathematische Zeitschrift},
    VOLUME = {72},
      YEAR = {1959/1960},
     PAGES = {192--204},
      ISSN = {0025-5874},
   MRCLASS = {10.00 (30.00)},
  MRNUMBER = {0117200},
MRREVIEWER = {N. G. de Bruijn},
       DOI = {10.1007/BF01162949},
       URL = {https://doi.org/10.1007/BF01162949},
}

\bib{Ro41}{article}{
    AUTHOR = {Rosser, J. B.},
     TITLE = {Explicit bounds for some functions of prime numbers},
   JOURNAL = {Amer. J. Math.},
  FJOURNAL = {American Journal of Mathematics},
    VOLUME = {63},
      YEAR = {1941},
     PAGES = {211--232},
      ISSN = {0002-9327},
   MRCLASS = {10.0X},
  MRNUMBER = {3018},
MRREVIEWER = {R. D. James},
       DOI = {10.2307/2371291},
       URL = {https://doi.org/10.2307/2371291},

}

\bib{Tr12}{article}{
    AUTHOR = {Trudgian, T. S.},
     TITLE = {An improved upper bound for the argument of the {R}iemann
              zeta-function on the critical line},
   JOURNAL = {Math. Comp.},
  FJOURNAL = {Mathematics of Computation},
    VOLUME = {81},
      YEAR = {2012},
    NUMBER = {278},
     PAGES = {1053--1061},
      ISSN = {0025-5718},
   MRCLASS = {11M06 (11M26)},
  MRNUMBER = {2869049},
MRREVIEWER = {Haseo Ki},
       DOI = {10.1090/S0025-5718-2011-02537-8},
       URL = {https://doi.org/10.1090/S0025-5718-2011-02537-8},
}

\bib{Tr14-1}{article}{
    AUTHOR = {Trudgian, T. S.},
     TITLE = {A new upper bound for {$|\zeta(1+it)|$}},
   JOURNAL = {Bull. Aust. Math. Soc.},
  FJOURNAL = {Bulletin of the Australian Mathematical Society},
    VOLUME = {89},
      YEAR = {2014},
    NUMBER = {2},
     PAGES = {259--264},
      ISSN = {0004-9727},
   MRCLASS = {11M06},
  MRNUMBER = {3182661},
MRREVIEWER = {St\'{e}phane R. Louboutin},
       DOI = {10.1017/S0004972713000415},
       URL = {https://doi.org/10.1017/S0004972713000415},
}

\bib{Tr14-2}{article}{
    AUTHOR = {Trudgian, T. S.},
     TITLE = {An improved upper bound for the argument of the {R}iemann
              zeta-function on the critical line {II}},
   JOURNAL = {J. Number Theory},
  FJOURNAL = {Journal of Number Theory},
    VOLUME = {134},
      YEAR = {2014},
     PAGES = {280--292},
      ISSN = {0022-314X},
   MRCLASS = {11M06 (11M26)},
  MRNUMBER = {3111568},
MRREVIEWER = {Haseo Ki},
       DOI = {10.1016/j.jnt.2013.07.017},
       URL = {https://doi.org/10.1016/j.jnt.2013.07.017},
}

\bib{Tr15}{article}{
    AUTHOR = {Trudgian, T. S.},
     TITLE = {An improved upper bound for the error in the zero-counting
              formulae for {D}irichlet {$L$}-functions and {D}edekind
              zeta-functions},
   JOURNAL = {Math. Comp.},
  FJOURNAL = {Mathematics of Computation},
    VOLUME = {84},
      YEAR = {2015},
    NUMBER = {293},
     PAGES = {1439--1450},
      ISSN = {0025-5718},
   MRCLASS = {11M06 (11M26 11R42)},
  MRNUMBER = {3315515},
MRREVIEWER = {Caroline L. Turnage-Butterbaugh},
       DOI = {10.1090/S0025-5718-2014-02898-6},
       URL = {https://doi.org/10.1090/S0025-5718-2014-02898-6},
}

\bib{vMo05}{article}{
    AUTHOR = {von Mangoldt, H. C. F.},
     TITLE = {Zur {V}erteilung der {N}ullstellen der {R}iemannschen
              {F}unktion {$\xi(t)$}},
   JOURNAL = {Math. Ann.},
  FJOURNAL = {Mathematische Annalen},
    VOLUME = {60},
      YEAR = {1905},
    NUMBER = {1},
     PAGES = {1--19},
      ISSN = {0025-5831},
   MRCLASS = {DML},
  MRNUMBER = {1511287},
       DOI = {10.1007/BF01447494},
       URL = {https://doi.org/10.1007/BF01447494},
}
       
\end{rezabib}

\end{document}